\documentclass[11pt]{article}
       \usepackage{amsfonts}
       \usepackage{stmaryrd}
       \usepackage{latexsym,amssymb,mathrsfs,fancyhdr}
       \font\tenmsb=msbm10
       \font\sevenmsb=msbm7
       \font\fivemsb=msbm5
       \catcode`\@=11
       \ifx\amstexloaded@\relax\catcode`\@=\active
       \endinput\else\let\amstexloaded@\relax\fi
       \def\spaces@{\space\space\space\space\space}
       \def\spaces@@{\spaces@\spaces@\spaces@\spaces@\spaces@}
       \def\space@.  {\futurelet\space@\relax}
       \space@.
       \def\Err@#1{\errhelp\defaulthelp@\errmessage{AmS-TeX error: #1}}
       \def\relaxnext@{\let\next\relax}
       \def\accentfam@{7}
       \def\noaccents@{\def\accentfam@{0}}
       \def\Cal{\relaxnext@\ifmmode\let\next\Cal@\else
       \def\next{\Err@{Use \string\Cal\space only in math mode}}\fi\next}
       \def\Cal@#1{{\Cal@@{#1}}}
       \def\Cal@@#1{\noaccents@\fam\tw@#1}
       \def\Bbb{\relaxnext@\ifmmode\let\next\Bbb@\else
       \def\next{\Err@{Use \string\Bbb\space only in math mode}}\fi\next}
       \def\Bbb@#1{{\Bbb@@{#1}}}
       \def\Bbb@@#1{\noaccents@\fam\msbfam#1}
       \newfam\msbfam
       \textfont\msbfam=\tenmsb
       \scriptfont\msbfam=\sevenmsb
       \scriptscriptfont\msbfam=\fivemsb
\usepackage{dsfont}
\usepackage[german,english]{babel}
\usepackage{amsmath,amssymb}
\usepackage[square, comma, sort&compress, numbers]{natbib}
\usepackage{cases}
\usepackage{mathrsfs}
\usepackage{amsmath}
\usepackage{amsthm}
\usepackage{amsfonts}
\usepackage{amssymb}
\usepackage{latexsym}
\usepackage{fancyhdr}
\usepackage{extarrows}
\usepackage{geometry}
\usepackage{color}
\usepackage[comma,numbers,square,sort&compress]{natbib}
\usepackage{epstopdf}
\usepackage{graphicx,amsmath}
\usepackage{subfigure}
\usepackage{amssymb}
\theoremstyle{plain}
\newtheorem{theorem}{Theorem}[section]
\newtheorem{corollary}[theorem]{Corollary}
\newtheorem{lemma}[theorem]{Lemma}

\numberwithin{equation}{section}
\newtheorem{remark}[theorem]{Remark}

\theoremstyle{definition}
\newtheorem{definition}[theorem]{Definition}

\large\normalsize

\usepackage{enumerate}

\begin{document}
\numberwithin{equation}{section}
\title{{\bf The $W$-weighted $m$-weak group MP inverse  \\ and its applications}}

\author{{\ Jiale Gao$^{a}$,  Kezheng Zuo$^{b}$\footnote{E-mail address: xiangzuo28@163.com (K. Zuo).}, Qing-Wen Wang$^{a}$}
\\ {{\small$^{a}$Department of Mathematics, Shanghai University, Shanghai 200444, China}}
\\{{\small $^b$School of Mathematics and Statistics, Hubei Normal University, Huangshi 435002, China}}}

\date{}

\maketitle

\begin{center}
\begin{minipage}{135mm}
{{\small {\bf Abstract:}
We extend the concept of the  $m$-weak group MP inverse of a square matrix to a rectangular matrix, called the   $W$-weighted $m$-weak group
MP inverse, which also unifies the $W$-weighted weak core inverse and  $W$-weighted DMP inverse.   Some properties, characterizations and representations of this new generalized inverse are shown.  Additionally, applications of the $W$-weighted weak group MP inverse are  given in solving a constrained optimization problem and  a class of consistent matrix equations.

\begin{description}
\item[{\bf Key words}:]$W$-weighted $m$-weak group
MP inverse, $m$-weak group MP inverse,  weighted core-EP decomposition, constrained optimization problem
\item[{\bf AMS subject classifications}:] 15A09, 15A24
\end{description}
}}
\end{minipage}
\end{center}

\section{Introduction}\label{introductionsection}
 Generalized inverses are studied in many branches of mathematical field: matrix theory, operator theory,  $C^*$-algebras, rings, etc. Additionally,  generalized inverses turn out to be a powerful tool in many applications such as  Markov chains,
differential equations, difference equations, electrical networks, optimal control, and so on. Refer to  \cite{Benbookre,Draganabookre,Wangbookre}.

A wave of research on generalized inverses has started  since Penrose \cite{PenroseMPinref} introduced the Moore-Penrose inverse in 1955.
In 1958, Drazin \cite{Drazininvref}  defined the Drazin inverse with spectral properties  in associative rings and semigroups.
Specifically, the Drazin inverse of a group matrix is also known as its  group inverse \cite{Erdegroupinversere}.
The core inverse as an alternative of  the group inverse is  presented   by Baksalary and  Trenkler  \cite{Bakscoreinversere} in 2010.
 Subsequently,   there  has been tremendous interest in the generalizations  of the core inverse, such as the BT inverse \cite{BaksalaryBTinref}, core-EP inverse \cite{PrasadcoreEPref}, DMP inverse \cite{MalikDMPiref}, weak core inverse \cite{Ferreyraweakcoreinref}, $m$-weak core inverse \cite{FerreyramWeakcoreref}.
  Recently, the $m$-weak group MP inverse  \cite{Jiangmwgmpiref} defined by we as a generalization of the core inverse also unifies the weak core  inverse and  DMP inverse.

As is well known,  the study of weighted generalized inverses is a significant research direction in the development history of generalized inverses, such as  the weighted Moore-Penrose inverse \cite{ChipmanWMPinref},  $W$-weighted Drazin inverse \cite{wdrazinclinere},
$W$-weighted BT inverse  \cite{FerreyraWBTinvref}, weighted core-EP inverse \cite{FerreyraWeCoerEPinref}, $W$-weighted DMP inverse \cite{MengWWDMPinvref}, and  $W$-weighted weak core inverse \cite{WWCoreDijre}.

  Inspired by the above facts, our  first research objective is to introduce the $W$-weighted $m$-weak group
MP inverse (in short, $W$-$m$-WGMP inverse), which serves as a generalization of the $m$-weak group MP inverse and unifies the $W$-weighted weak core inverse and  $W$-weighted DMP inverse.
We further show some of  its basic properties, characterizations and representations.

Generalized inverses have always played   a significant role in solving  the matrix equations.
  Wang et al. \cite{wangWGappref} recently discovered an interesting  result regarding the application of the weak group  inverse to the  least-squares solution of the constrained system of linear equations.   Motivated by the above result and the fact that $W$-weighted $m$-weak group  inverse \cite{WmWGinGaore} is a generalization of the weak group inverse, the second purpose of this paper is to extend this  result to the $W$-weighted $m$-weak group  inverse  and use its relationship with the $W$-$m$-WGMP inverse  to obtain the application of the   $W$-$m$-WGMP inverse in  a constrained optimization problem.   In addition, another  application of the   $W$-$m$-WGMP inverse are given in  solving  a class of consistent matrix equations.

 The paper is organized as follows.
 Some notations, definitions  and lemmas are recalled in Section \ref{Preliminaries}.  We devote Section \ref{weightedsection2} to defining the $W$-$m$-WGMP inverse and discussing  its basic properties, characterizations and representations.  In Section \ref{Applicationssection}, its applications  are considered  in solving the   matrix equations.  Conclusions are stated in Section \ref{Conclusionssection}.

\section{Preliminaries}\label{Preliminaries}

We now review some necessary notations, definitions  and lemmas.
Let $\mathbb{C}^{q\times n}$ and $\mathbb{N}^{+}$ be the set of all $q\times n$ complex matrices and the set of  all positive integers, respectively.   For $A\in\mathbb{C}^{q\times n}$, the symbols
$A^*$, $A^{-1}$, ${\rm Ind}(A)$,  ${\rm rank}(A)$, $\mathcal{R}(A)$ and $\mathcal{N}(A)$, ${\rm det}(A)$,  and $\Vert A\Vert_F$ denote the conjugate transpose, inverse ($q=n$), index ($q=n$),  rank, range, null space, determinant ($q=n$), and  Frobenius
norm  of $A$, respectively.
We  define that  $A^{0}=I_n$, where $A\in\mathbb{C}^{n\times n}$ and  $I_n$ is the identity matrix  of order $n$. And, $0$ denotes  the    null matrix of an appropriate size. The
projector onto  $\mathcal{T}$   along  $\mathcal{S}$ is denoted by  $P_{\mathcal{T},\mathcal{S}}$, where subspaces $\mathcal{T}$ and  $\mathcal{S}$   satisfy  their direct sum is equal to $\mathbb{C}^{n\times 1}$.
Particularly,  the orthogonal projector  onto $\mathcal{T}$ is $P_{\mathcal{T}}=P_{\mathcal{T},\mathcal{T}^{\perp}}$, where $\mathcal{T}^{\perp}$ presents the orthogonal complement of $\mathcal{T}$.  For $A\in\mathbb{C}^{q\times n}$ and a subspaces $\mathcal{W}$ in $\mathbb{C}^{n \times 1}$,  we define that  $A\mathcal{W}=\{Ax ~|~x\in\mathcal{W}\}$.
 The matrix obtained by replacing the $i$th column of $A\in\mathbb{C}^{q\times n}$  with $\beta \in  \mathbb{C}^{q\times 1} $ is denoted by $A({i \to {\beta}} )$,  where $i=1,2,...,n$.

  Throughout this paper,   $A^{D}$, $A^{\#}$, $A^{\textcircled{\dag}}$,  $A^{D,\dag}$, $A^{\textcircled{w}}$, $A^{\textcircled{w}_m}$, and $A^{\textcircled{w}_m,\dag}$ denote the Drazin inverse, group inverse, core-EP inverse, DMP inverse, weak group inverse, $m$-weak  group inverse \cite{wmweakgroupJiare}, and  $m$-weak  group MP inverse of $A\in\mathbb{C}^{n\times n}$,    respectively.  And, the symbols $A^{\dag}$, $A^{(2)}_{\mathcal{T},\mathcal{S}}$, $A^{D,W}$, $A^{\#,W}$, $A^{\textcircled{\dag},W}$, $A^{D,\dag}_w$, $A^{\textcircled{w},W}$, $A^{\textcircled{w},W,\dag}$,  and $A^{\textcircled{w}_m,W}$ represent the Moore-Penrose inverse, outer inverse with the range $\mathcal{T}$ and  null space $\mathcal{S}$, $W$-weighted Drazin inverse, $W$-weighted group inverse,  weighted  core-EP  inverse, $W$-weighted DMP inverse, $W$-weighted weak group inverse \cite{wweakgroupFerre}, $W$-weighted weak core inverse, and $W$-weighted $m$-weak group inverse of $A$, respectively,
where  $A\in\mathbb{C}^{q\times n}$, $W(\neq 0)\in\mathbb{C}^{n\times q}$ and  $m\in\mathbb{N}^{+}$. And, let $A\{2\}=\{X\in\mathbb{C}^{n\times q}~|~XAX=X\}$, for  $A\in\mathbb{C}^{q\times n}$.

Let $A\in\mathbb{C}^{q\times n}$, $B\in\mathbb{C}^{q\times n}$ and $W(\neq 0)\in\mathbb{C}^{n\times q}$.
The $W$-product \cite{FerreyraWeCoerEPinref} of $A$ and $B$   is defined   as $  A\star B= A W B$.
And,    we define
  the $W$-product of $A$ with itself $l\in\mathbb{N}^{+}$ times by $A^{\star l}$, i.e.,
\begin{equation*}
A^{\star l}=  \underbrace {A \star A\star... \star A}_{l \text{ times}} .
\end{equation*}
 Specially,  we define that   $A^{\star 0} W=I_{q}$ and $WA^{\star 0}  =I_{n}$.

\begin{lemma}[Weighted core-EP decomposition]\label{WcoreEPdecla}\cite[Theorem 4.1]{FerreyraWeCoerEPinref}
Let $A\in\mathbb{C}^{q\times n}$,   $W(\neq 0)\in\mathbb{C}^{n\times q}$ and $k=\max\{{\rm Ind}(AW),{\rm Ind}(WA)\}$.  Then,
\begin{equation}\label{AWdeceqs}
  A=U\left(
       \begin{array}{cc}
         A_1 & A_2 \\
         0 & A_3 \\
       \end{array}
     \right)
  V^* \text{ and }
  W=V\left(
       \begin{array}{cc}
        W_1 & W_2 \\
         0 & W_3 \\
       \end{array}
     \right)
  U^*,
\end{equation}
where $U\in\mathbb{C}^{q\times q}$ and  $V\in\mathbb{C}^{n\times n}$ are unitary matrices,
 $A_1\in\mathbb{C}^{t\times t }$ and $W_1\in\mathbb{C}^{t\times t }$ are nonsingular matrices,  $A_2\in\mathbb{C}^{t\times (n-t)}$, $W_2\in\mathbb{C}^{t\times (q-t)}$, and  $A_3W_3\in\mathbb{C}^{(q-t)\times (n-t) }$ and $W_3A_3\in\mathbb{C}^{(n-t) \times (q-t) }$ are nilpotent  of indices ${\rm Ind}(AW)$ and ${\rm Ind}(WA)$, respectively.
 \end{lemma}

\begin{lemma}[Weighted  Hartwig-Spindelb\"{o}ck decomposition]\label{wHSdeclemma}
Let $A\in\mathbb{C}^{q\times n}$ be of ${\rm rank}(A)=r_1$ and  $W(\neq 0)\in\mathbb{C}^{n\times q}$ be of  ${\rm rank}(W)=r_2$. Then,
\begin{equation}\label{wHSdecAWeq}
  A=T\left(
      \begin{array}{cc}
     \Sigma_1 K_1     &       \Sigma_1 L_1  \\
      0   & 0 \\
      \end{array}
    \right)S^*
    \text{ and }
    W=S\left(
      \begin{array}{cc}
    \Sigma_2 K_2     &       \Sigma_2 L_2  \\
      0   & 0 \\
      \end{array}
    \right)  T^*,
\end{equation}
where $T \in\mathbb{C}^{q\times q}$ and  $S \in\mathbb{C}^{n\times n}$ are  unitary matrices,  $ \Sigma_1 \in\mathbb{C}^{r_1\times r_1}$ and  $\Sigma_2 \in\mathbb{C}^{r_2\times r_2}$ are nonsingular matrices, and $K_1\in\mathbb{C}^{r_1\times r_1}$, $K_2 \in\mathbb{C}^{r_2\times r_2}$, $L_1\in\mathbb{C}^{r_1\times(n- r_1)}$ and $K_2 \in\mathbb{C}^{r_2\times(q- r_2)}$ are such that $K_1K_1^*+L_1L_1^*=I_{r_1}$ and $K_2K_2^*+L_2L_2^*=I_{r_2}$ .
\end{lemma}

\begin{lemma}\label{WmWGproperth}\cite[Theorems 3.3 and 3.4]{WmWGinGaore}
Let $A\in\mathbb{C}^{q\times n}$, $W(\neq 0)\in\mathbb{C}^{n\times q}$, $m\in\mathbb{Z}^{+}$ and $k=\max\{{\rm Ind}(AW),{\rm Ind}(WA)\}$. Then:
 \begin{enumerate}[$(1)$]
 \item\label{WmWGproperit00}
 $ A^{\textcircled{w}_m,W}=(A^{\textcircled{\dag},W})^{\star (m+1)}WA^{\star m}$;
   \item\label{WmWGproperit01} $ A^{\textcircled{w}_m,W}\in (WAW)\{2\}$;
   \item\label{WmWGproperit03} $\mathcal{R}(A^{\textcircled{w}_m,W})=\mathcal{R}((AW)^k)$;
   \item\label{WmWGproperit04} $\mathcal{N}(A^{\textcircled{w}_m,W})=\mathcal{N}(((WA)^k)^*WA^{\star m})$;
          \item\label{wmwgproAXit01} $ WAWA^{\textcircled{w}_m,W}=P_{\mathcal{R}\left((WA)^k\right),
       \mathcal{N}\left(((WA)^k)^*(WA)^m\right)}$;
   \item\label{wmwgproXAit02} $A^{\textcircled{w}_m,W}WAW=P_{\mathcal{R}\left((AW)^k\right),
       \mathcal{N}\left(((WA)^k)^*(WA)^{m+1}W\right)}$.
 \end{enumerate}
\end{lemma}

\begin{lemma}\cite[Theorem 5.2]{WeijicoreEPaproref}\label{WeijicoreEPaprolemma}
Let $A\in\mathbb{C}^{q\times n}$, $W(\neq 0)\in\mathbb{C}^{n\times q}$, $b\in\mathbb{C}^{n\times 1}$    and  $k=\max\{{\rm Ind}(AW),{\rm Ind}(WA)\}$.
Then, the constrained system:
 \begin{equation*}
  \mathop {\min }\limits_{\mathcal{R}(X)\subseteq \mathcal{R}((AW)^k)} \Vert WAW-b \Vert_2,
 \end{equation*}
 has the unique solution  $X=A^{\textcircled{\dag},W}b$, where $\Vert \cdot \Vert_2$ is the $2$-norm.
\end{lemma}

\begin{lemma}\cite[Theorem 3.2.8]{Wangbookre}\label{ATS2aappwangth}
Let $A\in\mathbb{C}^{q\times n}$ be such that $A^{(2)}_{\mathcal{T},\mathcal{S}}$ exits, where  $\mathcal{T} \subseteq \mathbb{C}^{n} $ and  $\mathcal{S} \subseteq \mathbb{C}^{q}  $ are subspaces. And, suppose that $G\in\mathbb{C}^{n\times q}$ satisfies that
$\mathcal{R}(G)=\mathcal{T}  $ and $\mathcal{N}(G)=\mathcal{S}  $ and that $V,U^*\in\mathbb{C}^{n\times (n-t)}$ satisfy that $\mathcal{R}(V)=\mathcal{N}(GA)$ and $\mathcal{N}(U)=\mathcal{R}(GA)$, where $t={\rm rank}(GA)$. Define $E=V(UV)^{-1}U$. Then,
$\mathcal{R}(E)=\mathcal{N}(GA)$,  $\mathcal{N}(E)=\mathcal{R}(GA)$,
$GA+E$ is nonsingular and $(GA+E)^{-1}=(GA)^{\#}+E^{\#}$.
\end{lemma}

\section{$W$-weighted $m$-weak group MP inverse}\label{weightedsection2}
In  order to extend the definition of  the $m$-weak group MP inverse to rectangular matrices, we first consider the  following system of matrix equations:
\begin{equation}\label{wmWGMPorteqs}
X=WA^{D,W}WAX,~AX=AWA^{\textcircled{w}_m,W}WAA^{\dag},
\end{equation}
where $A\in\mathbb{C}^{q\times n}$, $W(\neq 0)\in\mathbb{C}^{n\times q}$ and $m\in\mathbb{Z}^{+}$.

\begin{theorem}\label{origdenWWMMPinth}
The system \eqref{wmWGMPorteqs} has the unique solution $X=WA^{\textcircled{w}_m,W}WAA^{\dag}$.
\end{theorem}

\begin{proof}
Using  the fact that
\begin{equation}\label{BBDProeqeq}
 BB^{D}=P_{\mathcal{R}(B^d),\mathcal{N}(B^d)}
\end{equation}
 for $B\in\mathbb{C}^{n\times n}$ and $d\geq {\rm Ind}(B)$, and   \cite[Corollary 2.1]{wdrazinclinere}, i.e.,
 \begin{equation}\label{Wdrazinproit01}
   A^{D,W}=A((WA)^D)^2,
 \end{equation}
by   Lemma \ref{WmWGproperth}\eqref{WmWGproperit03} we have that
\begin{align*}
  X&=WA^{D,W}WAX=WA^{D,W}WAWA^{\textcircled{w}_m,W}WAA^{\dag}
  \\  &= WA((WA)^D)^2WAWA^{\textcircled{w}_m,W}WAA^{\dag}=
   (WA)^DWAWA^{\textcircled{w}_m,W}WAA^{\dag}\\
  &=  WA^{\textcircled{w}_m,W}WAA^{\dag},
\end{align*}
which completes the proof.
\end{proof}

\begin{definition}
Let $A\in\mathbb{C}^{q\times n}$, $W(\neq 0)\in\mathbb{C}^{n\times q}$ and $m\in\mathbb{Z}^{+}$.   The $W$-weighted weak group MP inverse (in short, $W$-$m$-WGMP inverse), denoted by $  A^{\textcircled{w}_m,W,\dag}$,  is defined as
\begin{equation}\label{WmWGMPdeneq}
  A^{\textcircled{w}_m,W,\dag}=WA^{\textcircled{w}_m,W}WAA^{\dag}.
\end{equation}
\end{definition}

\begin{remark}
Note that Gao et al.  in \cite[Corollary 3.8]{WmWGinGaore} have  proved that
$   A^{\textcircled{w}_m,W}=   A^{\textcircled{w},W}$  if $m=1$, $  A^{\textcircled{w}_m,W}=   A^{D,W}$ if $m\geq k$, and $  A^{\textcircled{w}_m,W}=  A^{\textcircled{w}_m}$ if $W=I_n$. And, in view of the definitions of the $W$-weighted weak core inverse \cite{WWCoreDijre}, $W$-weighted DMP inverse \cite{MengWWDMPinvref},  and $m$-weak  group MP inverse \cite{Jiangmwgmpiref},  i.e., $A^{\textcircled{w},W,\dag}=WA^{\textcircled{w},W}WAA^{\dag}$,
 $A^{D,\dag}_w=WA^{D,W}WAA^{\dag}$ and $A^{\textcircled{w}_m,\dag}=A^{\textcircled{w}_m}AA^{\dag}$,  by \eqref{WmWGMPdeneq}
it is clear   that $   A^{\textcircled{w}_m,W,\dag}=   A^{\textcircled{w},W,\dag}$  if $m=1$,  $  A^{\textcircled{w}_m,W}=   A^{D,\dag}_{w}$ if $m\geq k$,  and $  A^{\textcircled{w}_m,W}=   A^{\textcircled{w}_m}$ if $W=I_n$.
\end{remark}

We now   present some fundamental properties of the $W$-$m$-WGMP inverse  in order to facilitate further research.

 \begin{theorem}\label{WmWGWGMPproperth}
Let $A\in\mathbb{C}^{q\times n}$, $W(\neq 0)\in\mathbb{C}^{n\times q}$, $m\in\mathbb{Z}^{+}$ and $k=\max\{{\rm Ind}(AW),{\rm Ind}(WA)\}$. Then:
 \begin{enumerate}[$(1)$]
   \item\label{WmWGWGMPproperit01} $   A^{\textcircled{w}_m,W,\dag}\in A\{2\}$;
   \item\label{WmWGWGMPproperit02} ${\rm rank}(A^{\textcircled{w}_m,W,\dag})={\rm rank}{(WA^{D,W})}={\rm rank}{((WA)^{k})}$;
   \item\label{WmWGWGMPproperit03} $\mathcal{R}(A^{\textcircled{w}_m,W,\dag})=
       \mathcal{R}(WA^{D,W})= \mathcal{R}((WA)^{k})$;
   \item\label{WmWGWGMPproperit04} $\mathcal{N}(A^{\textcircled{w}_m,W,\dag})=
   \mathcal{N}(A^{\textcircled{\dag},W}(WA)^{m+1}A^{\dag})=
   \mathcal{N}(((WA)^k)^*(WA)^{m+1}A^{\dag})$;
   \item\label{WmWGWGMPproperit05} $A^{\textcircled{w}_m,W,\dag}=
       A^{(2)}_{\mathcal{R}(WA^{D,W}),\mathcal{N}(A^{\textcircled{\dag},W}(WA)^{m+1}A^{\dag})}=
       A^{(2)}_{\mathcal{R}((WA)^{k}),\mathcal{N}(((WA)^k)^*(WA)^{m+1}A^{\dag})}$.
 \end{enumerate}
\end{theorem}

 \begin{proof}
\eqref{WmWGWGMPproperit01}  Using \eqref{WmWGMPdeneq} and Lemma \ref{WmWGproperth}\eqref{WmWGproperit01}, we have that
\begin{align*}
   A^{\textcircled{w}_m,W,\dag}A  A^{\textcircled{w}_m,W,\dag}&=
   WA^{\textcircled{w}_m,W}WAA^{\dag}AWA^{\textcircled{w}_m,W}WAA^{\dag}\\
   &=    WA^{\textcircled{w}_m,W}WAWA^{\textcircled{w}_m,W}WAA^{\dag} =
    WA^{\textcircled{w}_m,W}WAA^{\dag}=A^{\textcircled{w}_m,W,\dag}.
\end{align*}
\par
\eqref{WmWGWGMPproperit02} According to \eqref{wmWGMPorteqs} and Lemma \ref{WmWGproperth}\eqref{wmwgproXAit02},
it follows that
\begin{align*}
 A   A^{\textcircled{w}_m,W,\dag}(AW)^{k+1}&=
 AWA^{\textcircled{w}_m,W}WAA^{\dag}(AW)^{k+1}\\&=
  AWA^{\textcircled{w}_m,W}WAW(AW)^{k }=
       (AW)^{k+1},
\end{align*}
which implies that
 $${\rm rank}{(W(AW)^{k})}\leq{\rm rank}{((AW)^{k})}=  {\rm rank}{((AW)^{k+1})}\leq {\rm rank}(A^{\textcircled{w}_m,W,\dag}).$$
Moreover, by \eqref{wmWGMPorteqs} and the fact that \cite[Lemma 2]{wdraziWeire}, i.e,
\begin{equation}\label{Wdrazinproit03}
  \mathcal{R}(A^{D,W})=\mathcal{R}((AW)^k),
\end{equation}
  it is obvious that
\begin{equation*}
 {\rm rank}(A^{\textcircled{w}_m,W,\dag}) \leq
  {\rm rank}{(WA^{D,W})}={\rm rank}{(W(AW)^{k})}.
\end{equation*}
Thus, $ {\rm rank}(A^{\textcircled{w}_m,W,\dag})={\rm rank}{(W(AW)^{k})}={\rm rank}{(WA^{D,W})}$. Note that
\begin{equation}\label{WAWkWAkraneq}
 {\rm rank}{(W(AW)^{k})} \leq {\rm rank}{((WA)^{k})}
 ={\rm rank}{((WA)^{k+1})}  \leq {\rm rank}{(W(AW)^{k})}.
\end{equation}
Hence, $ {\rm rank}(A^{\textcircled{w}_m,W,\dag})={\rm rank}{(W(AW)^{k})}={\rm rank}{((WA)^{k})} $.
\par
\eqref{WmWGWGMPproperit03}
Since $ \mathcal{R}(A^{\textcircled{w}_m,W,\dag}) \subseteq
       \mathcal{R}(WA^{D,W})$ follows from \eqref{wmWGMPorteqs},
using  the item    \eqref{WmWGWGMPproperit02},
\eqref{Wdrazinproit03} and  \eqref{WAWkWAkraneq}
 we deduce that
 \begin{equation}\label{WADrwakreqeq}
  \mathcal{R}(A^{\textcircled{w}_m,W,\dag}) =
       \mathcal{R}(WA^{D,W})=\mathcal{R}(W(AW)^{k})=\mathcal{R}((WA)^{k}).
\end{equation}
\par
\eqref{WmWGWGMPproperit04}
Note that
\begin{equation*}
 {\rm rank}{(W(AW)^{k})} \leq {\rm rank}{((AW)^{k})}
 ={\rm rank}{((AW)^{k+1})}  \leq {\rm rank}{(W(AW)^{k})},
\end{equation*}
implying that $ {\rm rank}{(W(AW)^{k})}={\rm rank}{((AW)^{k})}$. Then, by items \eqref{WmWGproperit03} and \eqref{WmWGproperit04} in Lemma \ref{WmWGproperth},      we have that
\begin{equation}\label{nwwmgqeqeq}
\mathcal{N}(WA^{\textcircled{w}_m,W})=
\mathcal{N}(A^{\textcircled{w}_m,W})=
\mathcal{N}(((WA)^k)^*WA^{\star m}).
\end{equation}
Thus, in terms of \eqref{WmWGMPdeneq}, we obtain that
\begin{align*}
 \mathcal{N}(A^{\textcircled{w}_m,W,\dag})&=
 \mathcal{N}(WA^{\textcircled{w}_m,W}WAA^{\dag})=
\left((WAA^{\dag})^*\mathcal{N}^{\perp}
(WA^{\textcircled{w}_m,W})\right)^{\perp}\\
&= \left((WAA^{\dag})^*\mathcal{N}^{\perp}(((WA)^k)^*WA^{\star m})\right)^{\perp}=\mathcal{R}^{\perp}\left((WAA^{\dag})^*(WA^{\star m})^*(WA)^k\right)\\
&= \mathcal{N}\left(((WA)^k)^*WA^{\star m}WAA^{\dag}\right)
=
   \mathcal{N}\left(((WA)^k)^*(WA)^{m+1}A^{\dag}\right).
\end{align*}
Similarly, by \cite[Proposition 3.1]{wcoreepigaore},
i.e.,    $\mathcal{N}(A^{\textcircled{\dag},W})=\mathcal{N}\left(((WA)^k)^*\right)$,  it follows that
$$\mathcal{N}\left(A^{\textcircled{\dag},W}(WA)^{m+1}A^{\dag}\right)=
 \mathcal{R}^{\perp}\left(((WA)^{m+1}A^{\dag})^*(WA)^k\right)=
   \mathcal{N}\left(((WA)^k)^*(WA)^{m+1}A^{\dag}\right).$$
 Hence, the item $(4)$ holds.
 \par
  \eqref{WmWGWGMPproperit05} It can be  directly obtained by items \eqref{WmWGWGMPproperit01},  \eqref{WmWGWGMPproperit03} and  \eqref{WmWGWGMPproperit04}.
\end{proof}

\begin{remark}
By comparing Theorem \ref{WmWGWGMPproperth}\eqref{WmWGWGMPproperit05} and \cite[Corollary 2.5]{WWCoreDijre}, i.e.,
$$A^{\textcircled{w},W,\dag}= A^{(2)}_{\mathcal{R}(WA^{D,W}),\mathcal{N}(A^{\textcircled{\dag},W}(WA)^2A^{\dag})}, $$
it is clear that $A^{\textcircled{w}_m,W,\dag}=A^{\textcircled{w},W,\dag}$ when $m=1$,
 and another outer inverse representation of the $W$-weighted weak core inverse is given as follows:
 $$A^{\textcircled{w},W,\dag}=
 A^{(2)}_{\mathcal{R}((WA)^{k}),\mathcal{N}(((WA)^k)^*(WA)^{2}A^{\dag})}.$$
\end{remark}

 \begin{theorem}\label{wwgwgmpoperTh}
Let $A\in\mathbb{C}^{q\times n}$, $W(\neq 0)\in\mathbb{C}^{n\times q}$, $m\in\mathbb{Z}^{+}$ and $k=\max\{{\rm Ind}(AW),{\rm Ind}(WA)\}$. Then:
 \begin{enumerate}[$(1)$]
   \item\label{wwgwgmpoperitem01} $  A A^{\textcircled{w}_m,W,\dag}=
       P_{\mathcal{R}\left((AW)^k\right),
       \mathcal{N}\left(((WA)^k)^*(WA)^{m+1}A^{\dag}\right)}$;
   \item\label{wwgwgmpoperitem02} $ A^{\textcircled{w}_m,W,\dag}A=P_{\mathcal{R}\left((WA)^k\right),
       \mathcal{N}\left(((WA)^k)^*(WA)^{m+1}\right)}$.
 \end{enumerate}
\end{theorem}

 \begin{proof}
$(1)$ Using Theorem \ref{WmWGWGMPproperth}\eqref{WmWGWGMPproperit03} and  \eqref{Wdrazinproit03},  we derive that
 \begin{equation*}
\mathcal{R}(  A A^{\textcircled{w}_m,W,\dag})=\mathcal{R}(AWA^{D,W})=
\mathcal{R}((AW)^{k+1})=\mathcal{R}((AW)^{k}).
\end{equation*}
Thus, by items  \eqref{WmWGWGMPproperit01} and \eqref{WmWGWGMPproperit04} in Theorem \ref{WmWGWGMPproperth}, it is clear that
\begin{equation*}
   A A^{\textcircled{w}_m,W,\dag}=
   P_{\mathcal{R}\left(A A^{\textcircled{w}_m,W,\dag}\right),
       \mathcal{N}\left(A^{\textcircled{w}_m,W,\dag}\right)}
   =
       P_{\mathcal{R}\left((AW)^k\right),
       \mathcal{N}\left(((WA)^k)^*(WA)^{m+1}A^{\dag}\right)}.
\end{equation*}
\par
$(2)$ According to \eqref{WmWGMPdeneq} and \eqref{nwwmgqeqeq}, it follows that
\begin{align*}
\mathcal{N}(A^{\textcircled{w}_m,W,\dag}A)
 &=\mathcal{N}(WA^{\textcircled{w}_m,W}WAA^{\dag}A)
=\mathcal{N}(WA^{\textcircled{w}_m,W}WA)
\\&=\left((WA)^*\mathcal{N}^{\perp}(WA^{\textcircled{w}_m,W})\right)^{\perp}
=\mathcal{R}^{\perp}\left((WA)^*(((WA)^k)^*WA^{\star m})^*  \right)
\\&= \mathcal{N}\left(((WA)^k)^*WA^{\star m}WA  \right)
=\mathcal{N}\left(((WA)^k)^*(WA)^{m+1} \right),
\end{align*}
which, together with items  \eqref{WmWGWGMPproperit01} and \eqref{WmWGWGMPproperit03} in Theorem  \ref{WmWGWGMPproperth}, shows that
\begin{equation*}
  A^{\textcircled{w}_m,W,\dag}A=
  P_{\mathcal{R}\left( A^{\textcircled{w}_m,W,\dag}\right),
       \mathcal{N}\left(A^{\textcircled{w}_m,W,\dag}A\right)}
  =P_{\mathcal{R}\left((WA)^k\right),
       \mathcal{N}\left(((WA)^k)^*(WA)^{m+1}\right)}.
\end{equation*}
This completes the proof.
\end{proof}

Different characterizations  and representations of the $W$-$m$-WGMP inverse are provided  in terms of  other generalized inverses,  such as the Moore-Penrose inverse, weighted core-EP inverse,  $W$-weighted Drazin inverse,  etc.

\begin{theorem}
Let $A\in\mathbb{C}^{q\times n}$, $W(\neq 0)\in\mathbb{C}^{n\times q}$ and $m\in\mathbb{Z}^{+}$.  Then the following conditions
are equivalent:
\begin{enumerate}[$(1)$]
  \item  $X=  A^{\textcircled{w}_m,W,\dag}$;
  \item $AX=(A^{\textcircled{\dag},W})^{\star m}WA^{\star m}WAA^{\dag}$ and $X=WA^{D,W}WAX$ ;
  \item  $AX=(A^{\textcircled{\dag},W})^{\star m}WA^{\star m}WAA^{\dag}$  and  $\mathcal{R}(X) \subseteq \mathcal{R}{(WA^{D,W})}$.
\end{enumerate}

\end{theorem}

\begin{proof}
$(1) \Rightarrow (2)$.
Using \cite[Theorems 2.2 and 2.3]{wcoreepigaore},  i.e.,
\begin{equation}\label{AWcoreEPeqWACepeq}
   A^{\textcircled{\dag},W}=A((WA)^{\textcircled{\dag}})^2,
\end{equation}
 we have that
\begin{align*}
  AWA^{\textcircled{\dag},W}WA^{\textcircled{\dag},W}=
  AWA((WA)^{\textcircled{\dag}})^2WA((WA)^{\textcircled{\dag}})^2
  = A((WA)^{\textcircled{\dag}})^2=A^{\textcircled{\dag},W}.
\end{align*}
Then, using \eqref{wmWGMPorteqs} and Lemma \ref{WmWGproperth}\eqref{WmWGproperit00},  we derive  that
\begin{align}
   AA^{\textcircled{w}_m,W,\dag}& =AWA^{\textcircled{w}_m,W}WAA^{\dag}=
  AW  (A^{\textcircled{\dag},W})^{\star (m+1)}WA^{\star m} WAA^{\dag}
  \nonumber\\
  &=AW A^{\textcircled{\dag},W}WA^{\textcircled{\dag},W} W (A^{\textcircled{\dag},W})^{\star (m-1)}WA^{\star m} WAA^{\dag}\nonumber\\
  &= A^{\textcircled{\dag},W} W (A^{\textcircled{\dag},W})^{\star (m-1)}WA^{\star m} WAA^{\dag}\nonumber\\
  &= (A^{\textcircled{\dag},W})^{\star m}WA^{\star m} WAA^{\dag}.\label{AXWGMPorcoreepeq}
\end{align}
Thus, ``$(1) \Rightarrow (2)$"    is clear by \eqref{wmWGMPorteqs}. \par
$(2) \Rightarrow (3)$.  It is apparent.
\par
$(3) \Rightarrow (1)$. By \eqref{WADrwakreqeq}, \eqref{BBDProeqeq} and \eqref{Wdrazinproit01}, it follows that
\begin{align*}
 \mathcal{R}(X) \subseteq \mathcal{R}{(WA^{D,W})}  & \Leftrightarrow
 \mathcal{R}(X) \subseteq \mathcal{R}{((WA)^k)}
  \Leftrightarrow   X=(WA)^DWAX\\
 & \Leftrightarrow X=WA((WA)^D)^2WAX
\Leftrightarrow X=WA^{D,W}WAX,
\end{align*}
which, together with \eqref{AXWGMPorcoreepeq} and \eqref{wmWGMPorteqs}, implies that the  item $(1)$ is true.
\end{proof}

 \begin{theorem}\label{WMGMPinrepth}
 Let $A\in\mathbb{C}^{q\times n}$, $W(\neq 0)\in\mathbb{C}^{n\times q}$, $m\in\mathbb{N}^{+}$ and  $k=\max\{{\rm Ind}(AW),{\rm Ind}(WA)\}$. Then:
\begin{enumerate}[$(1)$]
 \item\label{WMGMPinrepitem01}  $  A^{\textcircled{w}_m,W,\dag} = W\left((AW)^mA^{\textcircled{\dag},W}WA\right)^{\#,W}WA^{\star m}A^{\dag}$;
  \item\label{WMGMPinrepitem02} $  A^{\textcircled{w}_m,W,\dag} =((WA)^D)^{  (m+1)}P_{\mathcal{R}((WA)^k)}WA^{\star (m+1)}A^{\dag}$;
     \item\label{WMGMPinrepitem03}
$  A^{\textcircled{w}_m,W,\dag}=W A^{\star l}W(WA^{\star (l+m+1)}W)^{\dag}WA^{\star (m+1)}A^{\dag}
    $, where  $l\geq  k$;
  \item\label{WMGMPinrepitem04}  $ A^{\textcircled{w}_m,W,\dag}=W(WA^{\star (m+1)}WP_{\mathcal{R}((AW)^k)})^{\dag}WA^{\star (m+1)}A^{\dag} $;
  \item\label{WMGMPinrepitem05}  $ A^{\textcircled{w}_m,W,\dag}=WA^{\star (m-1)}W(A^{\star m})^{\textcircled{w},W}WAA^{\dag}$;
  \item\label{WMGMPinrepitem06}
  $ A^{\textcircled{w}_m,W,\dag}=((WA)^{\textcircled{w}})^{ m}WA^{\star m}A^{\dag}$;
  \item\label{WMGMPinrepitem07} $ A^{\textcircled{w}_m,W,\dag}=(WA)^{\textcircled{w}_m}WAA^{\dag}$.
\end{enumerate}

\begin{proof}
From \cite[Lemma 2.2(3)]{WmWGinGaore}, i.e., $A^{\textcircled{\dag},W}WAW=P_{\mathcal{R}((AW)^k),
        \mathcal{N}\left(((WA)^k)^{*}WAW\right)}$,
  and \cite[Proposition 3.1]{wcoreepigaore}, i.e., $\mathcal{R}(A^{\textcircled{\dag},W})=\mathcal{R}((AW)^k)$,  it is  clear that
  \begin{equation}\label{AWAWAAWcoreepeq}
   A^{\textcircled{\dag},W}WAW AW  A^{\textcircled{\dag},W}=AW  A^{\textcircled{\dag},W}.
  \end{equation}
By the fact that $(B^m)^{\#}={(B^{\#})}^m$ for $B\in\mathbb{C}^{n\times n} $ and ${\rm Ind}(B) \leq 1 $,  and \eqref{Wdrazinproit01}, it can be verified that
\begin{equation*}
  (A^{\#,W})^{\star m}=   (A ^{\star m})^{\#,W}.
\end{equation*}
Then, using \cite[Theorem 4.3(5)]{WmWGinGaore}, i.e.,
$A^{\textcircled{w}_m,W}=(A^{\textcircled{w},W})^{\star m}WA^{\star (m-1)}$, and \cite[Theorem 7]{wweakgroupFerre}, i.e., $A^{\textcircled{w},W}=(AWA^{\textcircled{\dag},W}WA)^{\#,W}$, by \eqref{WmWGMPdeneq} and  \eqref{AWAWAAWcoreepeq} we have that
\begin{align*}
A^{\textcircled{w}_m,W,\dag} =& WA^{\textcircled{w}_m,W}WAA^{\dag}
=W(A^{\textcircled{w},W})^{\star m}WA^{\star (m-1)}WAA^{\dag}=
 \\=&
 W((AWA^{\textcircled{\dag},W}WA)^{\#,W})^{\star m}WA^{\star m}A^{\dag}= W((AWA^{\textcircled{\dag},W}WA)^{\star m})^{\#,W}WA^{\star m}A^{\dag}
\\=& W((AW)^mA^{\textcircled{\dag},W}WA)^{\#,W}WA^{\star m}A^{\dag}.
\end{align*}
\par
The proof of items \eqref{WMGMPinrepitem02}--\eqref{WMGMPinrepitem07} can be obtained by \cite[Theorem 4.3]{WmWGinGaore} and \eqref{WmWGMPdeneq} immediately.
\end{proof}

\end{theorem}

\begin{remark}
The proof of Theorem \ref{WMGMPinrepth}\eqref{WMGMPinrepitem01}  indicates  a new expression for the $W$-weighted $m$-weak group inverse, namely,
\begin{equation*}
  A^{\textcircled{w}_m,W} = \left((AW)^mA^{\textcircled{\dag},W}WA\right)^{\#,W}WA^{\star (m-1)}.
\end{equation*}
Moreover, items \eqref{WMGMPinrepitem02}, \eqref{WMGMPinrepitem04}, \eqref{WMGMPinrepitem07}, and   \eqref{WMGMPinrepitem03} in Theorem \ref{WMGMPinrepth} extend \cite[Theorem 5.1 (a), (d) and (e)]{Jiangmwgmpiref}, and  \cite[Theorem 2.1]{DijanaWCIexref}, respectively.
\end{remark}

At the end of this section, we use the weighted core-EP decomposition and weighted  Hartwig-Spindelb\"{o}ck decomposition to give two   canonical forms of the $W$-$m$-WGMP inverse.

\begin{theorem}\label{WmWGWGMPdecTh}
Let $A\in\mathbb{C}^{q\times n}$ and  $W(\neq 0)\in\mathbb{C}^{n\times q}$ be given in \eqref{AWdeceqs}, and  let  $m\in\mathbb{Z}^{+}$ and $B_m=\sum\limits_{j = 0}^{m - 1} {{{({W_1}A{}_1)}^j}(W_1A_2+W_2A_3)(W_3A_3)^{m-1-j}}$. Then,
\begin{equation}\label{WmWGWGMPdecTeq}
  A^{\textcircled{w}_m,W,\dag} = V
  \left(
    \begin{array}{cc}
      A_1^{-1} &  (W_1A_1)^{-1}W_2A_3A_3^{\dag}+ (W_1A_1)^{-(m+1)}B_mW_3A_3A_3^{\dag}\\
     0 & 0 \\
    \end{array}
  \right)
  U^*.
\end{equation}
\end{theorem}

\begin{proof}
Using \eqref{AWdeceqs} and \cite[Theorem 3.7]{WmWGinGaore}, i.e.,
\begin{equation*}
A^{\textcircled{w}_m,W}=
U
\left(
  \begin{array}{cc}
   (W_1A_1W_1)^{-1}  & (A_1W_1)^{-(m+1)}W_1^{-1}B_m \\
    0 & 0 \\
  \end{array}
\right)
 V^*,
\end{equation*}
we deduce that
 \begin{align*}
 W A^{\textcircled{w}_m,W}W&=V\left(
       \begin{array}{cc}
        W_1 & W_2 \\
         0 & W_3 \\
       \end{array}
     \right)
\left(
  \begin{array}{cc}
   (W_1A_1W_1)^{-1}  & (A_1W_1)^{-(m+1)}W_1^{-1}B_m \\
    0 & 0 \\
  \end{array}
\right)
\left(
       \begin{array}{cc}
        W_1 & W_2 \\
         0 & W_3 \\
       \end{array}
     \right)
  U^*\\
  &=V\left(
  \begin{array}{cc}
   A_1^{-1}  & (W_1A_1)^{-1}W_2+(W_1A_1)^{-(m+1)}B_mW_3 \\
    0 & 0 \\
  \end{array}
\right)U^*.
\end{align*}
Then,  by \eqref{WmWGMPdeneq} and \cite[Theorem2.3]{WCMPMosre}, i.e.,
$  AA^{\dag}=U\left(
       \begin{array}{cc}
         I_t & 0 \\
         0 & A_3A^{\dag}_3 \\
       \end{array}
     \right)U^*,$
we have that
\begin{align*}
    A^{\textcircled{w}_m,W,\dag}&=WA^{\textcircled{w}_m,W}WAA^{\dag}\\
    &=
    V\left(
  \begin{array}{cc}
   A_1^{-1}  & (W_1A_1)^{-1}W_2+(W_1A_1)^{-(m+1)}B_mW_3 \\
    0 & 0 \\
  \end{array}
\right)U^*U\left(
       \begin{array}{cc}
         I_t & 0 \\
         0 & A_3A^{\dag}_3 \\
       \end{array}
     \right)U^*\\
     &= V
  \left(
    \begin{array}{cc}
      A_1^{-1} &  (W_1A_1)^{-1}W_2A_3A_3^{\dag}+ (W_1A_1)^{-(m+1)}B_mW_3A_3A_3^{\dag}\\
     0 & 0 \\
    \end{array}
  \right)
  U^*,
\end{align*}
which completes the proof.
\end{proof}

\begin{theorem}\label{WHSthnegenqeth}
Let $A\in\mathbb{C}^{q\times n}$ and  $W(\neq 0)\in\mathbb{C}^{n\times q}$ be given in \eqref{wHSdecAWeq}. Then,
\begin{equation*}
  A^{\textcircled{w}_m,W,\dag} =
 S \left(
    \begin{array}{cc}
    \Sigma_2K_2(\Sigma_1K_1)^{\textcircled{w}_m,\Sigma_2K_2}\Sigma_2K_2 & 0 \\
      0 & 0 \\
    \end{array}
  \right)T^*.
\end{equation*}

\end{theorem}

\begin{proof}
Evidently,
\begin{align*}
  WA&=S\left(
      \begin{array}{cc}
    \Sigma_2 K_2     &       \Sigma_2 L_2  \\
      0   & 0 \\
      \end{array}
    \right)  T^*T\left(
      \begin{array}{cc}
     \Sigma_1 K_1     &       \Sigma_1 L_1  \\
      0   & 0 \\
      \end{array}
    \right)S^*
    \\&=S\left(
      \begin{array}{cc}
    \Sigma_2 K_2  \Sigma_1 K_1     &         \Sigma_2 K_2 \Sigma_1 L_1  \\
      0   & 0 \\
      \end{array}
    \right) S^* .
\end{align*}
From \cite[Theorem 2.4]{wcoreepigaore}, i.e.,
\begin{equation*}
 A^{\textcircled{\dag},W}=  T
    \left(
    \begin{array}{cc}
  \Sigma_1K_1 ((\Sigma_2K_2\Sigma_1K_1)^{\textcircled{\dag}})^2 & 0
    \\
      0 & 0 \\
    \end{array}
  \right)S^*,
\end{equation*}
and \eqref{AWcoreEPeqWACepeq},
 it follows  that
\begin{align*}
 W  A^{\textcircled{\dag},W}& =
 S\left(
      \begin{array}{cc}
    \Sigma_2 K_2     &       \Sigma_2 L_2  \\
      0   & 0 \\
      \end{array}
    \right)T^* T
    \left(
    \begin{array}{cc}
  \Sigma_1K_1 ((\Sigma_2K_2\Sigma_1K_1)^{\textcircled{\dag}})^2 & 0
    \\
      0 & 0 \\
    \end{array}
  \right)S^*
  \\
  &=
 S \left(
    \begin{array}{cc}
  \Sigma_2 K_2  (\Sigma_1K_1)^{\textcircled{w}_m,\Sigma_2K_2}  & 0
    \\
      0 & 0 \\
    \end{array}
  \right)S^*.
\end{align*}
Then, using the fact that
$AA^{\dag}=T\left(
             \begin{array}{cc}
               I_{r_1} & 0 \\
               0 & 0 \\
             \end{array}
           \right)T^*
$, where $r_1={\rm rank }(A)$,  by \eqref{WmWGMPdeneq} and  Lemma \ref{WmWGproperth}\eqref{WmWGproperit00},   we have that
\begin{align*}
 A^{\textcircled{w}_m,W,\dag}&=WA^{\textcircled{w}_m,W}WAA^{\dag}
  =(  W A^{\textcircled{\dag},W} )^{m+1} (WA)^{m}WAA^{\dag}\\
  &=S \left(
    \begin{array}{cc}
  \Sigma_2 K_2  (\Sigma_1K_1)^{\textcircled{w}_m,\Sigma_2K_2}  & 0
    \\
      0 & 0 \\
    \end{array}
  \right)^{m+1}
  \left(
      \begin{array}{cc}
    \Sigma_2 K_2  \Sigma_1 K_1     &         \Sigma_2 K_2 \Sigma_1 L_1  \\
      0   & 0 \\
      \end{array}
    \right)^m \\
  & ~~~~ ~~\left(
      \begin{array}{cc}
    \Sigma_2 K_2     &       \Sigma_2 L_2  \\
      0   & 0 \\
      \end{array}
    \right)   \left(
             \begin{array}{cc}
               I_r & 0 \\
               0 & 0 \\
             \end{array}
           \right)T^*
  \\ &= S
  \left(
      \begin{array}{cc}
 \left( \Sigma_2 K_2  (\Sigma_1K_1)^{\textcircled{w}_m,\Sigma_2K_2}\right) ^{m+1} (\Sigma_2 K_2  \Sigma_1 K_1 )^m \Sigma_2 K_2     & 0 \\
      0   & 0 \\
      \end{array}
    \right)T^*\\
 &=S\left(
      \begin{array}{cc}
\Sigma_2 K_2  (\Sigma_1K_1)^{\textcircled{\dag},\Sigma_2K_2}
    \left(\Sigma_2K_2(\Sigma_1K_1)^{\textcircled{\dag},\Sigma_2K_2}\right)^m  (\Sigma_2 K_2  \Sigma_1 K_1 )^m \Sigma_2 K_2
    & 0 \\
      0   & 0 \\
      \end{array}
    \right)T^*
    \\&=S\left(
      \begin{array}{cc}
    \Sigma_2K_2(\Sigma_1K_1)^{\textcircled{w}_m,\Sigma_2K_2}\Sigma_2K_2 & 0 \\
      0 & 0 \\
      \end{array}
    \right)T^*,
\end{align*}
which completes  the proof.
\end{proof}

A new canonical form of the $m$-weak group MP inverse is given by Theorem \ref{WHSthnegenqeth} in the case $W=I_n$.

\begin{corollary}
Let $A\in\mathbb{C}^{n\times n}$ is given in \eqref{wHSdecAWeq} as $W=I_n$. Then,
\begin{equation*}
    A^{\textcircled{w}_m,\dag} =
 T \left(
    \begin{array}{cc}
 (\Sigma_1K_1)^{\textcircled{w}_m} & 0 \\
      0 & 0 \\
    \end{array}
  \right)T^*.
\end{equation*}
\end{corollary}

\section{Applications of the $W$-weighted $m$-weak group MP inverse }\label{Applicationssection}
Wang et al. in \cite[Theorem 1]{wangWGappref}  shows that
the following   problem:
 \begin{equation*}
  \mathop {\min }\limits_{\mathcal{R}(X)\subseteq \mathcal{R}((A)^k)} \Vert A^2X-AD \Vert_F,
 \end{equation*}
has the unique solution $ X=A^{\textcircled{w}}D$,  where $A\in\mathbb{C}^{n\times n}$, $D\in\mathbb{C}^{n\times p}$ and $k={\rm Ind}(A)$.
We begin this section by  extending the above result  to the $W$-weighted $m$-weak group  inverse. Consider the
  following
 constrained matrix approximation problem:
 \begin{equation}\label{WmGinappproeq}
  \mathop {\min }\limits_{{\mathcal{R}(X)\subseteq \mathcal{R}((AW)^k)}} \Vert WA^{\star (m+1)}WX-WA^{\star m}B \Vert_F,
 \end{equation}
where $A\in\mathbb{C}^{q\times n}$, $W(\neq 0)\in\mathbb{C}^{n\times q}$, $B\in\mathbb{C}^{q\times p}$, $m\in\mathbb{Z}^{+}$ and  $k=\max\{{\rm Ind}(AW),{\rm Ind}(WA)\}$.

\begin{theorem}\label{WGMGWMGMPappMIN01th}
 The problem \eqref{WmGinappproeq} has the unique solution $ X=A^{\textcircled{w}_m,W}B$.
\end{theorem}

\begin{proof}
For $d=\max\{{\rm Ind}(A^{\star (m+1)}W),{\rm Ind}(WA^{\star (m+1)})\}$,  one can verify that
\begin{equation*}
   \mathcal{R}(( A^{\star (m+1)}  W)^d) =   \mathcal{R}(A  W)^k).
\end{equation*}
 Let  $B= (b_1, b_2,..,b_p)$ and  $X= (x_1,x_2,..,x_p)$,
where $b_i,x_i\in\mathbb{C}^{q\times 1}$ for  $i=1,2,..,p$.  For each $i=1,2,...,p$,  using
\cite[Lemma 2.2(4)]{WmWGinGaore}, i.e., $(A^{\textcircled{\dag},W})^{\star m} = (A^{\star m}) ^{\textcircled{\dag},W}$,
by Lemma \ref{WeijicoreEPaprolemma} and Lemma \ref{WmWGproperth}\eqref{WmWGproperit00}, we deduce that the following
 \begin{equation*}
  \mathop {\min }\limits_{{\mathcal{R}(x_i)\subseteq \mathcal{R}((AW)^k)}} \Vert WA^{\star (m+1)}Wx_i-WA^{\star m}b_i \Vert_2,
 \end{equation*}
 has the unique solution
\begin{equation*}
    x_i= (A^{\star (m+1)})^{\textcircled{\dag},W}WA^{\star m}b_i= (A^{\textcircled{\dag},W})^{\star (m+1)}WA^{\star m}b_i=
  A^{\textcircled{w}_m,W}b_i.
\end{equation*}
Moreover, since
 \begin{align*}
 & \left(\Vert WA^{\star (m+1)}X-WA^{\star (m+1)}A^{\dag}B \Vert_F\right)^2
 \\ = & \left(\Vert WA^{\star (m+1)}(x_1,x_2,..,x_p)-WA^{\star (m+1)}A^{\dag}(b_1, b_2,..,b_p) \Vert_F\right)^2\\
 =  & \sum\limits_{i = 1}^p  \left(\Vert WA^{\star (m+1)} x_i- WA^{\star (m+1)}A^{\dag}b_i\Vert_2\right)^2,
  \end{align*}
the problem \eqref{WmGinappproeq} has the unique solution
\begin{equation*}
   X= \left(A^{\textcircled{w}_m,W}b_1, ~ A^{\textcircled{w}_m,W}b_2,~...,~A^{\textcircled{w}_m,W}b_p\right)=A^{\textcircled{w}_m,W}B.
\end{equation*}
This finishes the proof.
\end{proof}

In particular, we obtain a consequence of Theorem \ref{WGMGWMGMPappMIN01th} in the case $W=I_n$,  which is used to derive an application of the $W$-$m$-WGMP inverse  in solving the   constrained matrix approximation problem.

\begin{corollary}\label{WGWGMPsouapprofen01th}
Let $A\in\mathbb{C}^{n\times n}$,   $B\in\mathbb{C}^{n\times p}$, $m\in\mathbb{Z}^{+}$ and  $k={\rm Ind}(A)$.
Then the following
 constrained matrix approximation problem:
 \begin{equation*}
  \mathop {\min }\limits_{{\mathcal{R}(X)\subseteq \mathcal{R}(A^k)}} \Vert  A^{ m+1} X- A^{  m}B \Vert_F,
 \end{equation*}
has the unique solution
$ X=A^{\textcircled{w}_m}B$.
\end{corollary}

\begin{theorem}\label{WGMGWMGMPappMIN02th}
Let $A\in\mathbb{C}^{q\times n}$, $W(\neq 0)\in\mathbb{C}^{n\times q}$, $B\in\mathbb{C}^{q\times p}$, $m\in\mathbb{Z}^{+}$ and  $k=\max\{{\rm Ind}(AW),{\rm Ind}(WA)\}$.
  Then the following
 constrained matrix approximation problem:
  \begin{equation}\label{WAWAAmAminWGMPeq}
  \mathop {\min }\limits_{{\mathcal{R}(X)\subseteq \mathcal{R}((WA)^k)}} \Vert WA^{\star (m+1)}X-WA^{\star (m+1)}A^{\dag}B \Vert_F,
 \end{equation}
has the unique solution $ X=A^{\textcircled{w}_m,W,\dag} B$.
\end{theorem}

\begin{proof}
Note that
\begin{equation*}
WA^{\star (m+1)}X-WA^{\star (m+1)}A^{\dag}B=
   (WA)^{ (m+1)}X-(WA)^{ m}WAA^{\dag}B.
\end{equation*}
Then, applying Corollary  \ref{WGWGMPsouapprofen01th}  to the problem \eqref{WAWAAmAminWGMPeq}, by Theorem \ref{WMGMPinrepth}\eqref{WMGMPinrepitem07}
we derive that the problem  \eqref{WAWAAmAminWGMPeq}  has the
the unique solution
\begin{equation*}
   X=(WA)^{\textcircled{w}_m}WAA^{\dag}B=
   A^{\textcircled{w}_m,W,\dag}B,
\end{equation*}
which completes  the proof.
\end{proof}

A direct corollary of Theorem \ref{WGMGWMGMPappMIN02th}  for $W=I_n$ is given as follows.

\begin{corollary}\label{WGWGMPsouapproth}
Let $A\in\mathbb{C}^{n\times n}$,   $B\in\mathbb{C}^{n\times p}$, $m\in\mathbb{Z}^{+}$ and  $k={\rm Ind}(A)$.
 Then the following
 constrained matrix approximation problem:
  \begin{equation}\label{approximationjanMWGmpeq}
  \mathop {\min }\limits_{{\mathcal{R}(X)\subseteq \mathcal{R}(A^k)}} \Vert  A^{m+1}X- A^{m+1}A^{\dag}B \Vert_F,
 \end{equation}
has the unique solution
$ X=A^{\textcircled{w}_m,\dag} B$.
\end{corollary}

Following that,  we consider the following consistent  matrix equation:
\begin{equation}\label{WmWGMPappeq}
 ((WA)^k)^*(WA)^{m+1}X=((WA)^k)^*(WA)^{m+1}A^{\dag}B,
\end{equation}
where $A\in\mathbb{C}^{q\times n}$, $W(\neq 0)\in\mathbb{C}^{n\times q}$, $B\in\mathbb{C}^{q\times p}$, $m\in\mathbb{Z}^{+}$ and $k=\max\{{\rm Ind}(AW),{\rm Ind}(WA)\}$.
In the next theorem, we apply the $W$-$m$-WGMP inverse   to  represent  the general solution of Eq. \eqref{WmWGMPappeq}   as well as its unique solution under a certain constraint.

\begin{theorem}\label{wmwgmpapporgth}
 \begin{enumerate}[$(1)$]
   \item\label{wmwgmpapporgiem01} The general solution of Eq.  \eqref{WmWGMPappeq} is
   \begin{equation}\label{appeqsgensoleq}
     X=A^{\textcircled{w}_m,W,\dag}B+(I_n-A^{\textcircled{w}_m,W,\dag}A)Z,
   \end{equation}
   where $Z\in\mathbb{C}^{n\times m}$ is arbitrary.
   \item\label{wmwgmpapporgiem02} If the solution $X$ of  Eq. \eqref{WmWGMPappeq} satisfies $\mathcal{R}(X)\subseteq\mathcal{R}((WA)^k)$, then
   $X$ is unique and
   \begin{equation*}
   X=A^{\textcircled{w}_m,W,\dag}B.
   \end{equation*}
 \end{enumerate}

\end{theorem}

\begin{proof}
$(1)$ Write $Y=A^{\textcircled{w}_m,W,\dag}B+(I_n-A^{\textcircled{w}_m,W,\dag}A)Z$, where $Z\in\mathbb{C}^{n\times m}$.
Then, using Theorem \ref{wwgwgmpoperTh}\eqref{wwgwgmpoperitem02}, \eqref{WmWGMPdeneq} and Lemma \ref{WmWGproperth}\eqref{wmwgproAXit01}, we get that
\begin{align*}
 & ((WA)^k)^*(WA)^{m+1}Y
  \\=& ((WA)^k)^*(WA)^{m+1}\left(A^{\textcircled{w}_m,W,\dag}B+
  (I_n-A^{\textcircled{w}_m,W,\dag}A)Z\right)\\
=&((WA)^k)^*(WA)^{m}WAWA^{\textcircled{w}_m,W}WAA^{\dag}B+
  ((WA)^k)^*(WA)^{m+1}P_{\mathcal{N}\left(((WA)^k)^*(WA)^{m+1}\right),
\mathcal{R}\left((WA)^k\right)}Z\\
=&
 ((WA)^k)^*(WA)^{m}P_{\mathcal{R}\left((WA)^k\right),
       \mathcal{N}\left(((WA)^k)^*(WA)^m\right)}WAA^{\dag}B
=  ((WA)^k)^*(WA)^{m+1}A^{\dag}B.
\end{align*}
Using the facts that $G^{\textcircled{\dag}}=G^{D}G^l(G^l)^{\dag}$ for $G\in\mathbb{C}^{n\times n}$  and $l \geq Ind(G)$, and that $H^{\dag}=(H^*H)^{\dag}H^*$ for
$H\in\mathbb{C}^{q\times n}$, by  \eqref{AWcoreEPeqWACepeq}  we obtain that
\begin{align*}
WA^{\textcircled{\dag},W}(WA)^{m+1}A^{\dag}B& =
(WA)^{\textcircled{\dag}}(WA)^{m+1}A^{\dag}B=
(WA)^{D}(WA)^{k}((WA)^{k})^{\dag}(WA)^{m+1}A^{\dag}B\\
&=(WA)^{D}(WA)^{k}(((WA)^k)^*(WA)^k)^{\dag}((WA)^k)^*  (WA)^{m+1}A^{\dag}B \\&=
(WA)^{D}(WA)^{k}(((WA)^k)^*(WA)^k)^{\dag}((WA)^k)^* (WA)^{m+1}X\\
&=WA^{\textcircled{\dag},W}(WA)^{m+1}X,
\end{align*}
where $X$ is a solution of Eq. \eqref{WmWGMPappeq}.
Then,  by \eqref{WmWGMPdeneq} and   Lemma \ref{WmWGproperth}\eqref{WmWGproperit00},   it follows that
\begin{align*}
  A^{\textcircled{w}_m,W,\dag}B&=WA^{\textcircled{w}_m,W}WAA^{\dag}B
 =W(A^{\textcircled{\dag},W})^{\star (m+1)}WA^{\star m}WAA^{\dag}B\\
&=  W(A^{\textcircled{\dag},W})^{\star (m+1)}WA^{\star m}WAX
=WA^{\textcircled{w}_m,W}WAA^{\dag}AX=A^{\textcircled{w}_m,W,\dag}AX.
 \end{align*}
 Thus,
$
   X=A^{\textcircled{w}_m,W,\dag}B+(I_n-A^{\textcircled{w}_m,W,\dag}A)X.
$
Hence, the item \eqref{wmwgmpapporgiem01} holds.
\par
$(2)$ By  the item \eqref{wmwgmpapporgiem01} and Theorem \ref{WmWGWGMPproperth}\eqref{WmWGWGMPproperit03}, it is apparent that $ X=A^{\textcircled{w}_m,W,\dag}B$ is a solution of Eq. \eqref{WmWGMPappeq}  and
satisfies   $\mathcal{R}(X)\subseteq\mathcal{R}((WA)^k)$. To prove uniqueness,  let $X_i$ be such that
  \begin{align*}
 ((WA)^k)^*(WA)^{m+1}X_i&=((WA)^k)^*(WA)^{m+1}A^{\dag}B \text{ and }
 \mathcal{R}(X_i)\subseteq\mathcal{R}((WA)^k),
  \end{align*}
where $i=1,2$.
Obviously,
\begin{equation*}
 ((WA)^k)^*(WA)^{m+1}(X_1-X_2)=0,~
\mathcal{R}(X_1-X_2)\subseteq\mathcal{R}((WA)^k),
\end{equation*}
which, together with  Theorem \ref{wwgwgmpoperTh}\eqref{wwgwgmpoperitem02},  shows that
\begin{equation*}
  \mathcal{R}(X_1-X_2) \subseteq \mathcal{N}(((WA)^k)^*(WA)^{m+1}) \cap \mathcal{R}((WA)^k)=\{0\},
\end{equation*}
i.e., $X_1=X_2$.
\end{proof}

In the case $W=I_n$, by Theorem \ref{wmwgmpapporgth} we have the following corollary immediately.
\begin{corollary}\label{corWeqIcraxeqpr}
Let $A\in\mathbb{C}^{n\times n}$, $B\in\mathbb{C}^{n\times p}$, $m\in\mathbb{Z}^{+}$ and $k={\rm Ind}(A)$.
Consider the following matrix equation
\begin{equation}\label{mWGMPcoreq}
 (A^k)^*A^{m+1}X=(A^k)^*A^{m+1}A^{\dag}B.
\end{equation}
 \begin{enumerate}[$(1)$]
   \item The general solution of Eq.  \eqref{mWGMPcoreq} is
   \begin{equation*}
     X=A^{\textcircled{w}_m,\dag}B+(I_n-A^{\textcircled{w}_m,\dag}A)Z,
   \end{equation*}
   where $Z\in\mathbb{C}^{n\times m}$ is arbitrary.
   \item If the solution $X$ of  Eq. \eqref{mWGMPcoreq} satisfies $\mathcal{R}(X)\subseteq\mathcal{R}(A^k)$, then
   $X$ is unique and
   \begin{equation*}
   X=A^{\textcircled{w},\dag}B.
   \end{equation*}
 \end{enumerate}

 \end{corollary}

For the convenience of computation, we  first give the following result  based on Lemma \ref{ATS2aappwangth},  and then use it to obtain the Cramer's rule for the unique solution of the problem \eqref{WAWAAmAminWGMPeq} or Eq. \eqref{appeqsgensoleq}.

\begin{lemma}\label{UVEranginvth}
Let $A\in\mathbb{C}^{q\times n}$, $W(\neq 0)\in\mathbb{C}^{n\times q}$, $m\in\mathbb{Z}^{+}$,  $k=\max\{{\rm Ind}(AW),{\rm Ind}(WA)\}$ and $t={\rm rank}((WA)^k)$. And, let $V,U^*\in\mathbb{C}^{n\times (n-t)}$ be such that
$\mathcal{R}(V)=\mathcal{N}\left(((WA)^k)^*(WA)^{m+1}\right)$ and $\mathcal{N}(U)=\mathcal{R}\left((WA)^k\right)$.
Define
\begin{equation}\label{EUVeq}
  E=V(UV)^{-1}U.
\end{equation}
Then the following statements  are true:
\begin{enumerate}[$(1)$]
  \item\label{UVEranginvitem01} $\mathcal{R}(E)=\mathcal{N}\left(((WA)^k)^*(WA)^{m+1}\right)$
and
 $\mathcal{N}(E)=\mathcal{R}\left((WA)^k\right)$;
  \item\label{UVEranginvitem02}  $(WA)^k((WA)^k)^*(WA)^{m+1}+E$ is nonsingular and
  \begin{equation*}
    \left((WA)^k((WA)^k)^*(WA)^{m+1}+E\right)^{-1}
   =\left((WA)^k((WA)^k)^*(WA)^{m+1}\right)^{\#}+E^{\#}.
  \end{equation*}
\end{enumerate}

\end{lemma}

\begin{proof}
Let $G=(WA)^k((WA)^k)^*(WA)^{m+1}A^{\dag}$. By the fact $H^{\dag}B=P_{\mathcal{R}({H^*})}$ for $H\in\mathbb{C}^{q\times n}$ and Theorem \ref{wwgwgmpoperTh}\eqref{wwgwgmpoperitem02}, it follows that
 \begin{align*}
{\rm rank}\left((WA)^k\right)\geq&{\rm rank}(G) ={\rm rank}\left((WA)^k((WA)^k)^*(WA)^{m+1}A^{\dag}\right)
\\
\geq& {\rm rank}\left((WA)^k((WA)^k)^*(WA)^{m+1}A^{\dag}A\right)
  \\
=& {\rm rank}\left((WA)^k((WA)^k)^*(WA)^{m+1}\right)
\\
\geq&  {\rm rank}\left(((WA)^k)^{\dag}(WA)^k((WA)^k)^*(WA)^{m+1}\right)
\\
=&{\rm rank}\left(((WA)^k)^*(WA)^{m+1}\right)={\rm rank}\left((WA)^k\right).
\end{align*}
Then,
by items \eqref{WmWGWGMPproperit03} and \eqref{WmWGWGMPproperit04} in
Theorem \ref{WmWGWGMPproperth} and  Theorem \ref{wwgwgmpoperTh}\eqref{wwgwgmpoperitem02}, we have that
\begin{align*}
 & \mathcal{R}(G)= \mathcal{R}((WA)^{k})=\mathcal{R}( A^{\textcircled{w}_m,W,\dag}),\\
 &  \mathcal{N}(G)= \mathcal{N}(((WA)^k)^*(WA)^{m+1}A^{\dag})=\mathcal{N}( A^{\textcircled{w}_m,W,\dag}),\\
&\mathcal{R}(GA)= \mathcal{R}\left((WA)^k((WA)^k)^*(WA)^{m+1}\right)
=\mathcal{R}\left((WA)^k\right), \\
&\mathcal{N}(GA) =\mathcal{N}\left((WA)^k((WA)^k)^*(WA)^{m+1}\right)
 =\mathcal{N}\left(((WA)^k)^*(WA)^{m+1}\right).
\end{align*}
Finally, applying Lemma \ref{ATS2aappwangth}
 yields the items \eqref{UVEranginvitem01} and \eqref{UVEranginvitem02}.
\end{proof}

\begin{theorem}\label{awwcramemwgmpth}
Under the hypotheses of Theorem \ref{WGMGWMGMPappMIN02th}, let
$b_j$ be the j-th column vector of $((WA)^k)^*(WA)^{m+1} A^{\dag}B$, where  $j=1,2,...,p$.
Then the components of the unique solution of the problem \eqref{WAWAAmAminWGMPeq} (or Eq. \eqref{WmWGMPappeq} on the set  $\mathcal{R}((WA)^k)$)
are given by
\begin{equation}\label{componentsofXeq}
  x_{ij}= \dfrac{\det{\left(\left((WA)^k((WA)^k)^*(WA)^{m+1}+E\right)(i\rightarrow b_j)\right)}}
  {\det{\left((WA)^k((WA)^k)^*(WA)^{m+1}+E\right)}},
\end{equation}
where $E$ is given in \eqref{EUVeq}, $i=1,2,...,n$ and $j=1,2,...,p$.
\end{theorem}

\begin{proof}
We claim that $X=A^{\textcircled{w}_m,W,\dag}B$ is the unique solution of the following matrix equation
\begin{equation}\label{tranmaineqeq}
 \left((WA)^k((WA)^k)^*(WA)^{m+1}+E\right)X=
 (WA)^k((WA)^k)^*(WA)^{m+1}  A^{\dag}B.
\end{equation}
Indeed, by Lemma \ref{UVEranginvth}\eqref{UVEranginvitem01}, Theorem \ref{WmWGWGMPproperth}\eqref{WmWGWGMPproperit03}, \eqref{WmWGMPdeneq} and Lemma \ref{WmWGproperth}\eqref{wmwgproAXit01},  if follows that
\begin{align*}
 & \left((WA)^k((WA)^k)^*(WA)^{m+1}+E\right)A^{\textcircled{w}_m,W,\dag}B\\
 =&
 (WA)^k((WA)^k)^*(WA)^{m+1}A^{\textcircled{w}_m,W,\dag}B+EA^{\textcircled{w}_m,W,\dag}B
\\ =&(WA)^k((WA)^k)^*(WA)^{m}WAWA^{\textcircled{w}_m,W}WAA^{\dag}B\\
 =&(WA)^k((WA)^k)^*(WA)^{m}P_{\mathcal{R}\left((WA)^k\right),
       \mathcal{N}\left(((WA)^k)^*(WA)^m\right)}WAA^{\dag}B\\
  =&(WA)^k((WA)^k)^*(WA)^{m+1}A^{\dag}B,
\end{align*}
which, together with Lemma \ref{UVEranginvth}\eqref{UVEranginvitem02}, shows that $X=A^{\textcircled{w}_m,W,\dag}B$ is the unique solution of Eq. \eqref{tranmaineqeq}. Applying the  standard  Cramer's rule to Eq. \eqref{tranmaineqeq} yields \eqref{componentsofXeq} directly.
 \end{proof}

Similarly, Theorem \ref{awwcramemwgmpth} directly   gives the following corollary for $W=I_n$.

 \begin{corollary}
Under the hypotheses of Corollary \ref{corWeqIcraxeqpr}, the components of the unique solution of  the problem \eqref{approximationjanMWGmpeq} (or Eq. \eqref{mWGMPcoreq}  on the set  $\mathcal{R}(A^k)$) are given by
\begin{equation*}
  x_{ij}= \dfrac{\det{\left(\left(A^k(A^k)^* A^{m+1}+E\right)(i\rightarrow b_j)\right)}}
  {\det{\left(A^k( A^k)^* A^{m+1}+E\right)}},
\end{equation*}
where $E$ is given in \eqref{EUVeq} as $W=I_n$, $i=1,2,...,n$ and $j=1,2,...,p$.

 \end{corollary}

\section{Conclusions}\label{Conclusionssection}

We in this paper introduce a new generalized inverse, namely, the $W$-$m$-WGMP inverse, which not only generalises the $m$-weak group MP inverse but also unifies the  W-weighted weak core inverse and W-weighted DMP inverse.  Furthermore, we mainly discuss  properties, characterizations and representations of the $W$-$m$-WGMP inverse,  as well as its
  applications in solving the  least-squares solution of a constrained matrix equation and the general solution of a class of a  consistent matrix equations.     We are confident that  more explorations of the $W$-$m$-WGMP inverse will draw greater attention and interest  due to the broad range of research fields and application backgrounds of generalized inverses.  There are two research ideas for the $W$-$m$-WGMP inverse.

 \begin{enumerate}[$(1)$]
   \item Limit and integral representations, continuity,     perturbation analysis  and iterative methods for  the   $W$-$m$-WGMP inverse are all interesting study directions.

   \item
     It is possible  to  investigate  the   $W$-$m$-WGMP inverse of an operator between two Hilbert spaces as   a generalization of  the    weighted  $W$-$m$-WG inverse.

 \end{enumerate}


\begin{thebibliography}{10}

\bibitem{Bakscoreinversere}
O.M. Baksalary and G.~Trenkler.
\newblock Core inverse of matrices.
\newblock {\em Linear Multilinear Algebra}, 58(6):681--697, 2010.

\bibitem{BaksalaryBTinref}
O.M. Baksalary and G.~Trenkler.
\newblock On a generalized core inverse.
\newblock {\em Appl. Math. Comput.}, 236:450--457, 2014.

\bibitem{Benbookre}
A.~Ben-Israel and T.N.E. Greville.
\newblock {\em Generalized inverses: theory and applications (2nd edition)}.
\newblock Springer, New York, 2003.

\bibitem{ChipmanWMPinref}
J.S. Chipman.
\newblock On least squares with insufficient observations.
\newblock {\em J. Amer. Statist. Assoc.}, 54:1078--1111, 1964.

\bibitem{wdrazinclinere}
R.E. Cline and T.N.E. Greville.
\newblock A {D}razin inverse for rectangular matrices.
\newblock {\em Linear Algebra Appl.}, 29:53--62, 1980.

\bibitem{Draganabookre}
D.S. Cvetkovi\'{c}-Ili\'{c} and Y.~Wei.
\newblock {\em Algebraic properties of generalized inverses}.
\newblock Springer, Singapore, 2017.

\bibitem{Drazininvref}
M.P. Drazin.
\newblock Pseudo-inverses in associative rings and semigroups.
\newblock {\em Am. Math. Mon.}, 65(7):506--514, 1958.

\bibitem{Erdegroupinversere}
I.~Erdelyi.
\newblock On the matrix equation {$Ax = \lambda Bx$}.
\newblock {\em J. Math. Anal. Appl.}, 17(1):119--132, 1967.

\bibitem{Ferreyraweakcoreinref}
D.E. Ferreyra, F.E. Levis, and A.N.~Priori et~al.
\newblock The weak core inverse.
\newblock {\em Aequat. Math.}, 95:351--373, 2021.

\bibitem{FerreyraWeCoerEPinref}
D.E. Ferreyra, F.E. Levis, and N.~Thome.
\newblock Revisiting the core {E}{P} inverse and its extension to rectangular
  matrices.
\newblock {\em Quaest. Math.}, 41(2):265--281, 2018.

\bibitem{FerreyramWeakcoreref}
D.E. Ferreyra and S.B. Malik.
\newblock The $m$-weak core inverse.
\newblock {\em Rev. R. Acad. Cienc. Exactas F\'{\i}c. Nat. Ser. A Mat. RACSAM},
  118:article number 41, 2024.

\bibitem{wweakgroupFerre}
D.E. Ferreyra, V.~Orquera, and N.~Thome.
\newblock A weak group inverse for rectangular matrices.
\newblock {\em RACSAM}, 113:3727--3740, 2019.

\bibitem{FerreyraWBTinvref}
D.E. Ferreyra, N.~Thome, and C.~Torigino.
\newblock The {$W$}-weighted {BT} inverse.
\newblock {\em Quaest. Math.}, 46(2):359--374, 2023.

\bibitem{WmWGinGaore}
J.~Gao, K.~Zuo, and Q.~Wang.
\newblock A $m$-weak group inverse for rectangular matrices.
\newblock 2023. arXiv:2312.10704v1.

\bibitem{wcoreepigaore}
Y.~Gao, J.~Chen, and P.~Patr\'{i}cio.
\newblock Representations and properties of the {{$W$}}-weighted core-{E}{P}
  inverse.
\newblock {\em Linear Multilinear Algebra}, 68(6):1160--1174, 2020.

\bibitem{WeijicoreEPaproref}
J.~Ji and Y.~Wei.
\newblock The core-{EP}, weighted core-{EP} inverse of matrices and constrained
  systems of linear equations.
\newblock {\em Commun. Math. Res.}, 37(1):86--112, 2021.

\bibitem{Jiangmwgmpiref}
W.~Jiang, J.~Gao, X.~Zhang, and S.~Zuo.
\newblock $m$-weak group {MP} inverse.
\newblock 2024. Submitted.

\bibitem{wmweakgroupJiare}
W.~Jiang and K.~Zuo.
\newblock Further characterizations of the $m$-weak group inverse of a complex
  matrix.
\newblock {\em AIMS Mathe.}, 7(9):17369--17392, 2022.

\bibitem{MalikDMPiref}
S.B. Malik and N.~Thome.
\newblock On a new generalized inverse for matrices of an arbitrary index.
\newblock {\em Appl. Math. Comput.}, 226:575--580, 2014.

\bibitem{MengWWDMPinvref}
L.~Meng.
\newblock The {DMP} inverse for rectangular matrices.
\newblock {\em Filomat}, 31(19):6015--6019, 2017.

\bibitem{WCMPMosre}
D.~Mosi\'{c} and M.Z. Kolund\v{z}ija.
\newblock {W}eighted {CMP} inverse of an operator between {H}ilbert spaces.
\newblock {\em RACSAM}, 113:2155--2173, 2019.

\bibitem{WWCoreDijre}
D.~Mosi\'{c} and J.~Marovt.
\newblock Weighted weak core inverse of operators.
\newblock {\em Linear Multilinear Algebra}, 70(20):4991--5013, 2022.

\bibitem{DijanaWCIexref}
D.~Mosi\'{c} and P.S. Stanimirovi\'{c}.
\newblock Expressions and properties of weak core inverse.
\newblock {\em Appl. Math. Comput.}, 415:126704, 2022.

\bibitem{PenroseMPinref}
R.~Penrose.
\newblock A generalized inverse for matrices.
\newblock {\em Math. Proc. Cambridge}, 51(3):406--413, 1955.

\bibitem{PrasadcoreEPref}
K.~Manjunatha Prasad and K.S. Mohana.
\newblock Core-{E}{P} inverse.
\newblock {\em Linear Multilinear Algebra}, 62(6):792--802, 2014.

\bibitem{Wangbookre}
G.~Wang, Y.~Wei, and S.~Qiao.
\newblock {\em Generalized inverses: theory and computations (2nd edition)}.
\newblock Science Press, Beijing, 2018.

\bibitem{wangWGappref}
H.~Wang, J.~Gao, and X.~Liu.
\newblock The {WG} inverse and its application in a constrained matrix
  approximation problem.
\newblock {\em ScienceAsia}, 49:361--368, 2023.

\bibitem{wdraziWeire}
Y.~Wei.
\newblock A characterization for the {$W$}-weighted drazin inverse and a
  {C}ramer rule for the {$W$}-weighted {D}razin inverse solution.
\newblock {\em Appl. Math. Comput.}, 125:303--310, 2002.

\end{thebibliography}
\end{document}